\documentclass[12pt,leqno,twoside]{article}
\usepackage{amssymb}
\usepackage{amsmath}
\usepackage{amsthm}
\usepackage{a4,indentfirst,latexsym}
\usepackage{graphics}
\usepackage{mathrsfs}
\usepackage{cite,enumitem,graphicx}
\usepackage[colorlinks=true,urlcolor=black,
citecolor=black,linkcolor=black,linktocpage,pdfpagelabels,
bookmarksnumbered,bookmarksopen]{hyperref}
\usepackage[colorinlistoftodos]{todonotes}

\input xy
\xyoption{all}
\newtheorem{Thm}{Theorem}[section]
\newtheorem{Prop}[Thm]{Proposition}
\newtheorem{Lem}[Thm]{Lemma}

\newtheorem{Rem}[Thm]{Remark}
\newtheorem{Ex}[Thm]{Example}
\newtheorem{Ass}[Thm]{Assumption}

\newcommand\mytop[2]{\genfrac{}{}{0pt}{}{#1}{#2}}

\newcommand{\eps}{\varepsilon}

\def\id{\mathrm{id}}

\def\Z{\mathbb{Z}}

\def\N{\mathbb{N}}
\def\R{\mathbb{R}}
\def\C{\mathbb{C}}

\def\dim{\mathrm{dim}}

\newcommand{\cC}{{\mathcal C}}

\newcommand{\cF}{{\mathcal F}}

\newcommand{\cH}{{\mathcal H}}

\newcommand{\cM}{{\mathcal M}}

\newcommand{\ga}{\gamma}
\newcommand{\de}{\delta}
\newcommand{\ka}{\kappa}
\newcommand{\la}{\lambda}
\newcommand{\om}{\omega}
\newcommand{\si}{\sigma}
\newcommand{\Ga}{\Gamma}
\newcommand{\De}{\Delta}

\newcommand{\Om}{\Omega}
\newcommand{\Si}{\Sigma}

\newcommand{\ham}{H_\Omega}

\newcommand{\pa}{\partial}
\def\id{\mathrm{id}}

\newcommand{\beq[1]}{\begin{equation}\label{eq:#1}}
\newcommand{\eeq}{\end{equation}}

\numberwithin{equation}{section}

\DeclareMathOperator{\dist}{dist}
\DeclareMathOperator{\Ker}{Ker}
\DeclareMathOperator{\Range}{Range}
\DeclareMathOperator{\codim}{codim}
\newcommand\set[1]{\left\{\,#1\,\right\}}  
\newcommand\abs[1]{\left|#1\right|} 
\newcommand{\ska}[1]{\left\langle#1\right\rangle} 
\newcommand\norm[1]{\left\Vert#1\right\Vert} 

\begin{document}

\title{Periodic solutions of N-vortex type Hamiltonian systems near the domain boundary}
\author{Thomas Bartsch \and Qianhui Dai  \and Bj\"orn Gebhard}
\date{}
\maketitle

\begin{abstract}
The paper deals with the existence of nonstationary collision-free periodic solutions of singular first order Hamiltonian systems of $N$-vortex type in a domain $\Om\subset\C$. These are solutions $z(t)=(z_1(t),\dots,z_N(t))$ of
\[
\dot{z}_j(t)=-i\nabla_{z_j} \ham\big(z(t)\big),\quad j=1,\dots,N, \tag{HS}
\]
where the Hamiltonian $\ham$ has the form
\[
\ham(z_1,\dots,z_N)
 = -\sum_{\mytop{j,k=1}{j\ne k}}^N \frac{1}{2\pi}\log|z_j-z_k| -\sum_{j,k=1}^N g(z_j,z_k).
\]
The function $g:\Om\times\Om\to\R$ is required to be of class $\cC^3$ and symmetric, the regular part of a hydrodynamic Green function being our model. The Hamiltonian is unbounded from above and below, and the associated action integral is not defined on an open subset of the space of periodic $H^{1/2}$ functions.

Given a closed connected component $\Ga\subset\pa\Om$ of class $\cC^3$ we are interested in periodic solutions of \eqref{eq:HS} near $\Ga$. We present quite general conditions on the behavior of $g$ near $\Ga$ which imply that there exists a family of periodic solutions $z^{(r)}(t)$, $0<r<\overline{r}$, with arbitrarily small minimal period $T_r\to0$ as $r\to0$, and such that the ``point vortices'' $z_j^{(r)}(t)$ approach $\Ga$ as $r\to0$. The solutions are choreographies, i.e.\ $z_j^{(r)}(t)$ moves on the same trajectory as $z_1^{(r)}(t)$ with a phase shift. We can also relate the speed of each vortex with the curvature of $\Ga$.
\end{abstract}

{\bf MSC 2010:} Primary: 37J45; Secondary: 37N10, 76B47

{\bf Key words:} $N$-vortex dynamics; singular first order Hamiltonian systems; periodic solutions; choreography

\section{Introduction}\label{sec:intro}
The motion of $N$ point vortices in a planar domain $\Om\subset\C$ is described by a Hamiltonian system of the form
\begin{equation}\label{eq:HS}
\kappa_j\dot{z}_j(t)=-i\nabla_{z_j} \ham\big(z(t)\big),\quad j=1,\dots,N, \tag{HS}
\end{equation}
where $i\in\C$ is the imaginary unit; $z_j\in\Om$ is the position and $\kappa_j\in\R\setminus\{0\}$ the strength of the $j$-th vortex, respectively.

If $\Om$ is bounded and simply connected, the Hamiltonian $\ham$ is determined by the Green's function for the Dirichlet Laplacian
\[
G(w,z) = -\frac1{2\pi}\log|w-z| - g(w,z)\qquad\text{for }w,z\in\Om,\ w\ne z,
\]
and the so called Robin function $h$, which is the leading term of its regular part, i.e.
\[
h:\Om\to\R,\quad h(z)=g(z,z).
\]
To be more precise, one has for two or more vortices
\begin{equation}\label{eq:ham}
  \ham(z) = \sum_{\mytop{j,k=1}{j\ne k}}^N \ka_j\ka_k G(z_j,z_k)
             - \sum_{k=1}^N\ka_k^2h(z_k),
\end{equation}
defined on
\[
\cF_N(\Om)=\{(z_1,\cdots,z_N)\in\Om^N: z_j\neq z_k\,\,\textrm{for}\,\,j\neq k\}.
\]
The motion of a single vortex in $\Om$ is completely determined by the Robin function, because then $\ham=-\ka_1^2h$. The Hamiltonian $\ham$, the so-called Kirchhoff-Routh path function, arises when making a point vortex ansatz for an incompressible non-viscous fluid described by the 2D-Euler equations. This goes back to Kirchhoff \cite{kirchhoff:1876}, Routh \cite{routh:1881} and C.C.~Lin \cite{lin:1941a,lin:1941b}. The work of Lin covers also the case of multiply connected domains, in which the Dirichlet Green's function is replaced by more general hydrodynamic Green's functions. It should be mentioned that the Hamiltonian as in \eqref{eq:ham} corresponds to the case when $\pa\Om$ is a solid boundary, i.e.\ there is no flux through the boundary. One can find the derivation of $\ham$ and more aspects of vorticity methods also in recent literature, e.g.\ \cite{flucher-gustafsson:1997,majda-bertozzi:2001,marchioro-pulvirenti:1994,newton:2001,saffman:1992}.

We would also like to mention that \eqref{eq:HS} with Hamiltonians having the same logarithmic term as in \eqref{eq:ham} plus possibly different regular summands appear in several other applications from mathematical physics. Examples are the Ginzburg-Landau-Schr\"odinger equation or the Landau-Lifshitz-Gilbert equation; see\cite{Colliander-Jerrard:1998, Kurzke-etal:2011} and references therein.

A major difficulty in all these settings is that, except in a few special cases, the Hamiltonian is unbounded from above and below, singular, not integrable, and has noncompact, not even metrically complete level surfaces. Therefore even the existence of equilibria is difficult to prove, and results have been achieved only quite recently; see \cite{bartsch-pistoia:2014,bartsch-pistoia-weth:2010,delpino-etal:2005,kuhl:2015,kuhl:2016}. Next to equilibria the existence and properties of periodic solutions are of great importance, of course. If $\Om=\C$ or if $\Om$ is radially symmetric like a disk or an annulus, there are many results about special periodic solutions of \eqref{eq:HS};
see \cite{aref-etal:2002,newton:2001} and the references therein. For general domains very little is known. It is even unknown whether one can apply the generalized version of the Weinstein-Moser theorem \cite{Bartsch:1997, moser:1976, Weinstein:1973} to find periodic solutions near the equilibria from \cite{bartsch-pistoia:2014,bartsch-pistoia-weth:2010,delpino-etal:2005,kuhl:2015,kuhl:2016}. Also the existence of periodic solutions on fixed energy surfaces seems to be out of reach because these are not compact, not even complete as metric spaces. Standard methods from critical point theory do not apply either nor does Floer theory, so the existence of periodic solutions for the $N$-vortex problem in an arbitrary domain is at present far from being solved.

For non-stationary periodic solutions of \eqref{eq:HS}, the only results we are aware of are \cite{bartsch:2016,bartsch-dai:2016,bartsch-gebhard:2016,bartsch-saccet:2016}. In \cite{bartsch-dai:2016} a local family of periodic solutions oscillating around a stable critical point of the Robin function $h$ with arbitrarily small minimal periods has been obtained in case $\ka_1=\cdots=\ka_N$. The solutions are choreographic, i.e.\ all vortices $z_j$ move on the same trajectory with a phase shift. After a suitable scaling they look like the classical $N$-gon configuration of Thomson. This result has been significantly improved in \cite{bartsch-gebhard:2016}, where, firstly, general vorticities $\ka_1,\dots,\ka_N\ne0$ are allowed; secondly, the shape of the vortex configurations near the equilibrium resembles after rescaling a prescribed relative equilibrium in the plane, not just Thomson's $N$-gon; and thirdly, these solutions lie on global connected continua of periodic solutions. In \cite{bartsch:2016} the local result has been refined in case the critical point of $h$ is nondegenerate. These results hold in generic bounded domains due to \cite{caffarelli-friedman:1985,micheletti-pistoia:2014}. Another type of periodic solutions has been found in \cite{bartsch-saccet:2016} for the 2-vortex problem. There the two vortices rotate around their center of vorticity, while the center of vorticity itself is close to a level line of the Robin function $h$, i.e.\ a trajectory of the 1-vortex system. Also these solutions exist in generic bounded domains.

In the present paper we find a new type of periodic choreographies for \eqref{eq:HS} with $H$ as in \eqref{eq:ham}. Whereas the results from\cite{bartsch:2016,bartsch-dai:2016,bartsch-gebhard:2016,bartsch-saccet:2016} yield periodic solutions where all vortices are close to each other, here we obtain periodic solutions where the vortices are far away from each other. The vortices move on a trajectory close to a prescribed connected component $\Ga$ of the domain boundary $\pa\Om$, separated by phase shifts. Of course, if $N\to\infty$ then the vortices will be close to their nearest neighbors. For simplicity we only treat the case of identical vortices but we do allow arbitrarily many vortices, a general, even unbounded domain, and we consider general functions $g$ in \eqref{eq:ham}.

The behavior of point vortices close to solid boundaries is of interest in oceanographic modelling, for instance. There the system \eqref{eq:HS} serves as a basic model for the interaction of ocean eddies with a given topography, e.g.\ islands in front of a coastline. In the papers \cite{crowdy-marshall:2005,crowdy-marshall:2006} Crowdy and Marshall use conformal mappings between multiply connected domains to compute trajectories of a single vortex passing various types of coast-island configurations. Another example is the influence of a bay embedded into a straight coastline on the motion of a point vortex, which is investigated in \cite{ryzhov-koshel:2016} by Ryzhov and Koshel. Our results are of a different nature in that we consider the interaction of possibly many vortices with a bounded coastline, and we find periodic solutions.

The article is organized as follows. We state and discuss our results in Section \ref{sec:results}. In Section~\ref{sec:prelim} we present the main idea of the proof and reformulate the problem as an equation between Hilbert spaces. The proof of the main theorem is then given in Section \ref{sec:solving}. Lastly we prove in Section~\ref{sec:simply_connected_domains} that the theorem applies in particular to simply connected bounded domains with $\cC^{3,\alpha}$ boundary.

\section{Statement of results}\label{sec:results}

Let $\Om\subset\C$ be a domain with nonempty boundary $\pa\Om$, and let $\Ga\subset\pa\Om$ be a closed connected component of the boundary of class $\cC^3$. Clearly $\Ga$ is diffeomorphic to $S^1$. Let $L$ denote the length of $\Ga$, $\nu:\Ga\to\R^2$ the exterior unit normal, and $\ka:\Ga\to\R$ the curvature of $\Ga$ with respect to $\nu$. Set $d(z):=\dist(z,\Ga)$ and fix $\de>0$ sufficiently small such that the orthogonal projection
\[
p:\,\,\Om_\de:=\{z\in\Om: d(z)\le\de\}\to\Ga
\]
is well defined. We may assume that $|\ka(p)|<\frac1{\de}$ for all $p\in\Ga$. Consider a symmetric $\cC^3$ function $g:\Om\times\Om\to\R$, and set
\[
  G(w,z) = -\frac1{2\pi}\log|w-z| - g(w,z)\qquad\text{for }w,z\in\Om,\ w\ne z.
\]
We also need the functions $h,\rho:\Om\to\R$ defined by $h(z)=g(z,z)$ and $\rho(z)=\exp(-2\pi h(z))$. In the classical point vortex case $G$ is the Green function of the Dirichlet Laplace operator in $\Om$, $h$ the Robin function, and $\rho$ the harmonic radius (see \cite{flucher:1999}). We require the following behavior of these functions near $\Ga$.

\begin{Ass}\label{ass}
  $\rho$ can be extended to a function $\rho\in\cC^3(\Om\cup\Ga)$ by setting $\rho(q)=0$ for $q\in\Ga$. Moreover, $\nabla\rho(q)=-2\nu(q)$, and $\rho''(q)=-2\ka(q)\cdot\id$ for every $q\in\Ga$. For every $\eps>0$ the function $G$ satisfies
  \[
    \abs{\nabla_1 G(w,z)}+\abs{\nabla_1^2 G(w,z)} = O(d(z)),\quad\text{and}\quad
    \nabla_2\nabla_1 G(w,z) = O(1)\nu(p(z))^T+O(d(z))
  \]
  as $d(z)\to 0$ uniformly on the set $\set{(w,z)\in\overline{\Om}\times\Om_\de:\abs{w-z}\ge \eps}$.
\end{Ass}

It follows from assumption~\ref{ass} that $\rho(z)=2d(z)-\ka(p(z))d(z)^2+o(d(z)^2)$ as $d(z)\to0$, and  $h(z)\to\infty$ as $d(z)\to0$. We define the bundle
\[
  \cH^1 = \bigcup_{T>0}\{T\}\times H^1(\R/T\Z,\C^{N})
\]
containing all periodic functions having a square-integrable derivative and parametrized by the period. We equip $\cH^1$ with the chart $\cH^1\to \R\times H^1(\R/2\pi\Z,\C^N)$, $(T,u)\mapsto \left(T,u\big(\cdot\frac{T}{2\pi}\big)\right)$. Now we are ready to state our main theorem about the Hamiltonian system \eqref{eq:HS} with Hamiltonian
\[
  \ham(z) = \sum_{\mytop{j,k=1}{j\ne k}}^N G(z_j,z_k) - \sum_{k=1}^N h(z_k).
\]

\begin{Thm}\label{thm:main}
  If assumption~\ref{ass} holds then there exists $\overline{r}>0$ and a $\cC^1$ map
  \[
  (0,\overline{r})\to \cH^1,\quad r\mapsto\left(T_r\hspace{2pt},\hspace{2pt}z^{(r)}\right)=\left(2\pi rL, \big(z_1^{(r)},\cdots,z_N^{(r)}\big)\right)
  \]
  such that $z^{(r)}$ is a periodic solution of \eqref{eq:HS} with minimal period $T_r=2\pi rL$ for each $r$. Moreover, these periodic solutions possess the following properties:

  \noindent (1) $z_k^{(r)}(t)=z_1^{(r)}\left(t+\frac{(k-1)T_r}{N}\right)$ for every $k=1,\dots,N$.

  \noindent (2) The rescaled function $v^{(r)}(t):=z_1^{(r)}(2\pi rt)$ converges in $H^1(\R/L\Z,\C)$ as $r\to0$ towards a parametrization $\gamma$ of $\Ga$ according to arc-length. More precisely, setting $\nu=\nu\circ\gamma$, $\kappa=\kappa\circ\gamma$ by abuse of notation, there holds
  \[
  	v^{(r)}=\gamma-r\nu+o(r)\quad \text{in }H^1(\R/L\Z,\C),
  \]
  and
  \[
  	\dot{v}^{(r)}=(1-r\kappa)i\nu+o(r)\quad \text{uniformly in }t\in\R.
  \]

  \noindent (3) The distance $d(v^{(r)}(t))$ satisfies
  \[
    d(v^{(r)}) = r+\frac{1}{2}\kappa r^2+o(r^2)
  \]
  as $r\to0$ uniformly in $t\in\R$.
\end{Thm}
 \begin{Rem}\rm
a) The theorem shows that for $T>0$ small enough, the system \eqref{eq:HS} has a $T$-periodic solution with all vortices moving on the same trajectory. At first order the trajectory has distance $r=T/2\pi L$ from $\Ga$. The second order term in (3) tells us that the trajectory of the vortices comes closer to $\Ga$ in regions where $\Ga$ has negative curvature, and the vortices speed up by (2). On the other hand, near positively curved parts of $\Ga$ the trajectory increases the distance to $\Ga$ and the vortices slow down. In any case the vortices try to use shortcuts near curved parts of the boundary.

b) Observe that all conclusions in Theorem~\ref{thm:main} require only that $\Ga$ and $g$ are of class $\cC^2$. Therefore the theorem can be generalized to the case where $\Ga$, $g$ and the extension of $\rho$ in Assumption~\ref{ass} are only of class $\cC^2$ by approximation.

c) It is clear that the factor 2 in $\nabla\rho(q)=-2\nu(q)$ and $\rho''(q)=-2\ka(q)\cdot\id$ from assumption~\ref{ass} can be replaced by any $c>0$ independent of $q\in\Ga$. If $c=c(q)$ depends on $q\in\Ga$ this will effect the motion of the vortices significantly.
\end{Rem}

\begin{Ex}\rm
As an illustration of our result consider the case of the unit disk $\Om=B_1(0)$. In this case the Green's and Robin function are explicitly known (see \eqref{eq:Greens_function_ball} below). Periodic solutions of $N$ vortices of equal strength arranged in a polygonal configuration exist (not only) near the boundary, those are,
\[
z_k(t)=se^{i\om(s) t}e^{i\frac{2\pi(k-1)}{N}}, \quad k=1,\cdots,N,
\]
with the radius $s\in(0,1)$ and uniform angular velocity
\[
\om(s) = \frac{1}{\pi s^2}\left(\frac{N}{1-s^{2N}}-\frac{N+1}{2}\right).
\]
This solution has the minimal period $2\pi rL=4\pi^2 r$ if
\[
  r=\frac{s^2(1-s^{2N})}{N-1+(N+1)s^{2N}}.
\]
This is satisfied for $s\sim 0$ and $s\sim 1$. Close to $\pa B_1(0)$, i.e.\ for $s\sim 1$, one can verify that  $d(z)- r-\frac{1}{2}\kappa r^2=(1-s)- r-\frac{1}{2}r^2=o(r^2)$ as $r\to0$. This coincides with our theorem.
\end{Ex}

We expect that Assumption~\ref{ass} holds for $g$ being the regular part of a hydrodynamic Green function in any domain with a bounded component $\Ga\subset\pa\Om$ of class $\cC^3$ but haven't found this in the literature. Here we prove it in the simply connected case.

\begin{Prop}\label{prop:simply_connected}
  Assumption \ref{ass} holds for $g$ being the regular part of the Green function of Dirichlet Laplace operator in a simply connected bounded domain with $\cC^{3,\alpha}$ boundary.
\end{Prop}

\section{Preliminaries}\label{sec:prelim}

\subsection{The Scaling idea}\label{sec:scal}

In order to obtain periodic solutions of \eqref{eq:HS} close to $\Ga$ we introduce a suitable scaling. For $0<r<\de$ we define the diffeomorphism $\chi_r:\Om_\de\to\Om_r:=\{z\in\Om:d(z)\leq r\}$ by
\[
  \chi_r(z)=\frac{r}{\de}z+\left(1-\frac{r}{\de}\right)p(z)=p(z)-\frac{r}{\de}d(z)\nu(p(z)).
\]
Clearly $\chi_r$ pushes $z$ to $p(z)$ as $r\to0$. Then $z(t)=(z_1(t),\ldots,z_N(t))\in \cF_N(\Om_r)$ is a solution of \eqref{eq:HS} if and only if $u(t)=(u_1(t),\ldots,u_N(t))$ defined by
\[
  u_k(t):=\chi_r^{-1}(z_k(2\pi rt))
\]
is a solution of the system
\begin{equation}\label{eq:HSr}
  \dot{u}_k = 2\pi rD\chi_r^{-1}(\chi_r(u_k))[-i\nabla_{z_k}\ham(\chi_r(u_1),\cdots,\chi_r(u_N))],
  \quad k=1,\cdots,N.
  \tag{${\textrm{HS}}_r$}
\end{equation}
On a suitable set of $(u_1,\dots,u_n)$ we will see that \eqref{eq:HSr} posseses a limiting system \eqref{eq:HS0}, whose solutions can be extended to solutions of \eqref{eq:HSr} for $r>0$ small.

\begin{Lem}\label{lem:deri} For $z\in\Om_\de$ we write $\nu=\nu(p(z))$, $\ka=\ka(p(z))$, $\ka'=D\ka(p(z))[i\nu(p(z))]$, by abuse of notation. Then one has
\begin{gather*}
  \nabla d=-\nu,\quad Dp=\frac{1}{1-d\ka}i\nu (i\nu)^T,\quad D\nu=\ka Dp,\quad \nabla\ka=\frac{\ka'}{1-d\ka}i\nu,\\
  D\chi_r(z)=\frac{r}{\de}\nu\nu^T+\frac{\de-rd\ka}{\de(1-d\ka)}i\nu(i\nu)^T.
\end{gather*}
\end{Lem}

\begin{proof} The derivatives of $d$ and $\nu$ are of course known and only listed here for convenience. Obviously we have $Dp(z)[v]=0$ for $v\in \R\nu(p(z))$. If $v\in \R i\nu(p(z))$, differentiation of $d(z)\nu(p(z))=p(z)-z$ yields
\[
  d(z)\ka(p(z)) Dp(z)[v]=Dp(z)[v]-v
\]
so the derivative of $p$ follows. The expression for $\nabla \ka$ follows directly. Finally, the definition of  $\chi_r$ yields:
\[
  D\chi_r(z) = \frac{r}{\de}\cdot\id+\left(1-\frac{r}{\de}\right)Dp
             = \frac{r}{\de}\nu\nu^T+\frac{\de-rd\ka}{\de(1-d\ka)}i\nu(i\nu)^T.
\]
\end{proof}

In the presence of only one vortex with strength $\ka_1=1$, the Hamiltonian $\ham$ is given by $\ham(z)=-h(z)=\frac{1}{2\pi}\log\rho(z)$. Thus the scaled system \eqref{eq:HSr} becomes
\begin{equation}\label{HSr1}
  \dot{u} = 2\pi rD\chi_r^{-1}(\chi_r(u))\big[i\nabla h(\chi_r(u)\big].
  \tag{${\textrm{HS}}^1_r$}
\end{equation}
Lemma \ref{lem:deri} implies
\[
  rD\chi_r^{-1}(\chi_r(u)) = \de\nu\nu^T+ r\de\frac{1-d\ka}{\de-rd\ka}i\nu(i\nu)^T,
\]
where $\nu$ and $\ka$ are evaluated at $p(u)$ and $d$ at $u$. Moreover, Assumption \ref{ass} gives
\begin{equation}\label{eq:gradient_of_h}
  \nabla h(\chi_r(u)) = -\frac{1}{2\pi}\frac{\nabla\rho(\chi_r(u))}{\rho(\chi_r(u))}
                      = \frac{\de\nu-rd\ka\nu+o(r)}{2\pi rd+o(r)}
\end{equation}
as $r\to 0$. Thus, the limit system of \eqref{HSr1} at $r=0$ is
\begin{equation}\label{eq:HS0}
  \dot{u} = \big(1-d(u)\ka(p(u))\big)\frac{\de}{d(u)} i\nu(p(u)).
  \tag{${\textrm{HS}}_0$}
\end{equation}

\begin{Prop}\label{prop:period}
Any solution $u$ of the limit system \eqref{eq:HS0} is periodic and has constant distance $d(u)$ to $\Ga$. If $\eps\equiv d(u)\in(0,\de]$ denotes the distance, then the period of $u$ is $\frac{\eps}{\de}L$.
\end{Prop}

Recall that $L$ denotes the length of $\Ga$.

\begin{proof} Indeed, any solution $u(t)$ of \eqref{eq:HS0} satisfies
\[
\frac{d}{dt}d(u(t))=\ska{-\nu,\frac{(1-d\ka)\de}{d}i\nu}_{\R^2}=0
\]
and thus moves on a closed curve with constant distance to $\Ga$, in particular it is periodic. Let $\eps=d(u(t))$ be the distance and $T_\eps>0$ be the minimal period of $u$. Since $p\circ u:\R/T_\eps\mathbb{Z}\to\Ga$ is
a parameterization of $\Ga$, there holds:
\[
\begin{aligned}
L &= \int_0^{T_\eps}\big|p(u(t))'\big|\,dt
   =\int_0^{T_\eps}\bigg|\frac{1}{1-d(u(t))\ka(p(u(t)))}\dot{u}(t)\bigg|\,dt\\
  &=\int_0^{T_\eps}\bigg|\frac{\de i\nu(p(u(t)))}{d(u(t))}\bigg|\,dt
   =\frac{\de T_\eps}{\eps}.
\end{aligned}
\]
\end{proof}
\vspace{-5mm}
We denote by $u^0$ the solution of \eqref{eq:HS0} having distance $d(u^0)\equiv \de$ and therefore period $L$.

Now we turn to the case of $N\in\N$ identical vortices. Let $\Si_N$ be the permutation group of $N$ symbols, which acts orthogonally on $\R^{2N}=\C^N$ via
$$
  \si *(z_1,\ldots,z_N)=(z_{\si^{-1}(1)},\ldots,z_{\si^{-1}(N)}),\quad \si\in\Si_N.
$$
Since $\ham$ is invariant under the action of $\Si_N$, we can look for a special type of choreographic solutions of \eqref{eq:HSr}, namely solutions $U(t)$ that satisfy
$$
  U(t+L/N)=\si^{-1}*U(t),
$$
where $\si=(1~2~\ldots~N)\in\Si_N$ is the cyclic permutation. The ansatz $U=(u,u(\cdot+\theta_2),\ldots,u(\cdot+\theta_N))$, with $\theta_j=(j-1)\frac{L}{N}$, $j=1,\ldots,N$,
reduces our problem to finding $L$-periodic solutions $u$ of
\begin{equation}\label{eq:reduce}
  \dot{u}
    =2\pi rD\chi_r^{-1}(\chi_r(u))
      \Big[-i\nabla_{z_1}\ham\big(\chi_r(u),\chi_r(u(\cdot+\theta_2)),\ldots,\chi_r(u(\cdot+\theta_N))\big)\Big].
\tag{${\textrm{HS}}_r'$}
\end{equation}
Indeed, the invariance of $\ham$ under $\Si_N$ implies
\begin{equation}\label{eq:symmetry_of_nabla_H}
  \nabla_{1}\ham(\si^{-(k-1)}*z) = \nabla_{\si^{k-1}(1)}\ham(z)=\nabla_k\ham(z),\quad k=1,\ldots,N.
\end{equation}
Consequently $U$ solves \eqref{eq:HSr} provided $u$ is a $L$-periodic solution of \eqref{eq:reduce}. It remains to solve \eqref{eq:reduce} for $r\sim 0$.

\subsection{The Hilbert space setting}

Let $L^2=L^2_L(\C)=L^2(\R/L\Z,\C)$ be the Hilbert space of $L$-periodic square integrable functions with scalar
product
\[
  \langle u,v\rangle_{L^2}=\int_0^L\langle u(t),v(t)\rangle_{\R^2}dt=\int_0^L\textrm{Re}\big(\overline{u(t)}v(t)\big)dt
\]
and associated norm $\|\cdot\|_{L^2}$. Furthermore let $H^1=H^1_L(\C)=H^1(\R/L\Z,\C)$ denote the Sobolev space of $L$-periodic, absolutely continuous functions having a square integrable derivative equipped with the scalar product
\[
\langle u,v\rangle=\langle u,v\rangle_{L^2}+\langle \dot{u},\dot{v}\rangle_{L^2}
\]
and induced norm $\|\cdot\|$. The group $S^1=\R/L\Z$ acts on both spaces via time shift, i.e.
\[
\theta\ast u(t)=u(t+\theta),\quad \theta\in S^1.
\]
Finding an $L$-periodic solution of \eqref{eq:reduce} is equivalent to finding a zero of
\begin{equation}\label{eq:Fru}
  F(r,u) = \dot{u}+2\pi rD\chi_r^{-1}(\chi_r(u))
            \Big[i\nabla_{z_1}\ham\big(\chi_r(u),\chi_r(\theta_2*u),\ldots,\chi_r(\theta_N*u)\big)\Big],
\end{equation}
defined on a subset of $\R\times H^1$ and mapping into $L^2$. Recall that our solution $u^0$ of the one-vortex system \eqref{eq:HS0} has the constant distance $d(u^0)\equiv\de$ to the boundary component $\Ga$. The projection $p$ onto $\Ga$ is also defined on a bounded open neighbourhood $V$ of $\Om_\de$. Since $H^1$ embeds into the space of continuous functions, we can fix a constant $\de_1>0$, such that for all $u\in H^1$ with $\norm{u-u^0}\leq \de_1$ the following holds: $u(t)\in  V$, $d(u(t))\geq \de_1$, and
\[
  \abs{p(\theta_k *u(t))-p(\theta_j*u(t))}\geq \de_1 ,\quad t\in\R,1\le k < j\le N.
\]
There exist $r_0>0$ and $c>0$ satisfying $\chi_r(u(t))\in V$ and
\begin{equation}\label{eq:Fwelldefined}
  \abs{\chi_r(\theta_k*u(t))-\chi_r(\theta_j*u(t))}\geq c
\end{equation}
whenever $r\in(0,r_0)$, $u\in H^1$ with $\norm{u-u^0}\leq\de_1$, $t\in\R$ and $1\le k < j\le N$. We may also assume that $\de_1<\norm{u^0}$. Then $F(r,\cdot)$ is defined for $r\in(0,r_0)$ on the tubular neighborhood $B_{\de_1}(\cM)\subset H^1$ of the $S^1$-orbit $\cM:=S^1*u^0$.

\begin{Lem}\label{lem:FC1}
The map $F:[0,r_0)\times B_{\de_1}(\cM)\to L^2$ defined for $r>0$ as in \eqref{eq:Fru}, and for $r=0$ by
$$
  F(0,u) = \dot{u}-\big(1-d(u)\ka(p(u))\big)\frac{\de}{d(u)} i\nu(p(u)) = \dot{u}-\frac{(1-d\ka)\de}{ d}i\nu,
$$
is of class $\cC^1$ with derivatives
$$
D_uF(0,u)[w] = \dot{w}+\de\left(\frac{\ka}{d}\nu\ska{i\nu,w}_{\R^2}
                 +\frac{\ka'}{1-d\ka}i\nu\ska{i\nu,w}_{\R^2}-\frac{1}{ d^2}i\nu\ska{\nu,w}_{\R^2}\right),
$$
and
$$
  \pa_rF(0,u)=\frac{1}{2}\ka'd\nu-\frac{1}{2}(1-d\ka)\ka i\nu.
$$
\end{Lem}

\begin{proof} For $r\in(0,r_0)$ we have
$$
\begin{aligned}
  F(r,u) &= \dot{u} + 2\pi rD\chi_r^{-1}(\chi_r(u))
              \Big[i\sum_{j=2}^N\,2\nabla_1G(\chi_r(u),\chi_r(\theta_j\ast u))-i\nabla h(\chi_r(u_1))\Big]\\
         &= \dot{u}+F_1(r,u)+2\sum_{j=2}^NF_j(r,u)
\end{aligned}
$$
where we set
$$
\begin{aligned}
  \psi(r,u) &:= rD\chi_r^{-1}(\chi_r(u))=\de\nu\nu^T+ r\de\frac{1-d\ka}{\de-rd\ka}i\nu(i\nu)^T,\\
   F_1(r,u) &:= -2\pi\psi(r,u)\big[i\nabla h(\chi_r(u_1))\big],\\
   F_j(r,u) &:= 2\pi\psi(r,u)\big[i\nabla_1G(\chi_r(u),\chi_r(\theta_j\ast u))\big],\quad j=2,\ldots,N.
\end{aligned}
$$
It is clear that $F$ restricted to $(0,r_0)\times B_{\de_1}(\cM)$ is $\cC^1$. From case the $N=1$ in Section~\ref{sec:scal} we already know that
$$
  F_1(r,u)=-\frac{(1-d\ka)\de}{d}i\nu+o(1)
$$
as $r\to 0$. This holds in $L^2$ and uniformly in $u\in B_{\de_1}(\cM)$ due to  \ref{ass}. Regarding $F_j$ for $j\ge 2$ Assumption~\ref{ass} gives
$$
F_j(r,u)=O(d(\chi_r(\theta_j*u)))=O(r),
$$
because $\chi_r(u(t))$ and $\chi_r(\theta_j*u(t))$ are uniformly bounded away from each other by \eqref{eq:Fwelldefined}. This proves the continuity of $F$. Next a straightforward calculation shows
\begin{equation}\label{eq:Dupsi}
D_u\psi(r,u)[w]=\alpha\left(i\nu\nu^T+\nu(i\nu)^T\right)+r\beta (i\nu)(i\nu)^T,
\end{equation}
with
\begin{equation}\label{eq:definition_of_alpha}
\alpha=\frac{(\de-r)\ka\de}{(1-d\ka)(\de-rd\ka)}\ska{i\nu,w}_{\R^2},\quad\beta=\frac{(\de-r)\de}{(\de-rd\ka)^2}\left(\ka\ska{\nu,w}_{\R^2}-\frac{d\ka'}{1-d\ka}\ska{i\nu,w}_{\R^2}\right).
\end{equation}
So again by Assumption~\ref{ass} we have for $j\ge2$:
$$
D_uF_j(r,u)[w]=O(r)+O(1)\nu(p(\theta_j*u))^TD\chi_r(\theta_j*u)[\theta_j*w]=O(r)
$$
uniformly in $u$.

For the partial derivative with respect to $r$ we use the symmetry of $G$ to deduce for $j\ge2$:
$$
  \begin{aligned}
  \pa_rF_j(r,u)
    &= O(r)+O(1)\cdot(i\nu)^T\nabla_2\nabla_1 G(\chi_r(u),\chi_r(\theta_j*u))\nu(p(\theta_j*u))\\
    &= O(r)+O(1)\cdot\nu(p(\theta_j*u))^T\nabla_2\nabla_1 G(\chi_r(\theta_j*u),\chi_r(u))i\nu\\
    &= O(r).
\end{aligned}
$$

Concerning $F_1$ we have
\begin{equation}\label{eq:calc0}
\begin{aligned}
  D_uF_1(r,u)[w] &= D_u\psi(r,u)[w]\left[i\frac{\nabla\rho(\chi_r(u))}{\rho(\chi_r(u))}\right]\\
  &\hspace{1cm}
    +\psi(r,u)\left[i\frac{\rho''(\chi_r(u))}{\rho(\chi_r(u))}
       -i\frac{\nabla\rho(\chi_r(u))\nabla\rho(\chi_r(u))^T}{\rho(\chi_r(u))^2}\right]D\chi_r(u)[w].
\end{aligned}
\end{equation}
Using \eqref{eq:Dupsi} and expansions for $\rho$, $\nabla\rho$ according to \ref{ass} one derives for the first term
\begin{equation}\label{eq:calc1}
  D_u\psi(r,u)[w]\left[i\frac{\nabla\rho(\chi_r(u))}{\rho(\chi_r(u))}\right]
   = -\frac{2\alpha}{\rho(\chi_r(u))}\nu+\alpha\ka\nu-\frac{\beta\de}{d}i\nu+o(1).
\end{equation}
For the second term observe that with some $\xi=\xi(u(t))\in \left(0,\frac{r}{\de}\right)$:
$$
\begin{aligned}
  (i\nu)^T\rho''(\chi_r(u))(i\nu)
    &=-2\ka+\rho'''(p-\xi d\nu)[-\nu,i\nu,i\nu]rd\\
    &= -2\ka+\pa_{i\nu}(\rho''(p)[-\nu,i\nu])rd+o(r) = -2\ka+o(r),
\end{aligned}
$$
which implies
\begin{equation}\label{eq:calc2}
  \psi(r,u)\left[i\frac{\rho''(\chi_r(u))}{\rho(\chi_r(u))}\right]D\chi_r(u)[w]
    = \frac{2\ka(\de-rd\ka)}{1-d\ka}\ska{i\nu,w}_{\R^2}\frac{\nu}{\rho(\chi_r(u))}+o(1).
\end{equation}
Next
\begin{equation}\label{eq:calc3}
  \psi(r,u)\left[-i\frac{\nabla\rho(\chi_r(u))\nabla\rho(\chi_r(u))^T}{\rho(\chi_r(u))^2}\right]D\chi_r(u)[w]
    = -\frac{\de^2(1-d\ka)}{d^2(\de-rd\ka)}\ska{\nu,w}_{\R^2}i\nu+o(1).
\end{equation}
Combining \eqref{eq:calc0}-\eqref{eq:calc3}, the first order expansion of $\rho$ and the definition of $\alpha,\beta$ in \eqref{eq:definition_of_alpha} yields
$$
\begin{aligned}
  D_uF_1(r,u)[w]
    &= \left(\frac{2\ka(\de-rd\ka)}{1-d\ka}\ska{i\nu,w}_{\R^2}-2\alpha\right)\frac{\nu}{\rho(\chi_r(u))}+\alpha\ka\nu\\
    &\hspace{40pt}
        -\left(\frac{\beta\de}{d}+\frac{\de^2(1-d\ka)}{d^2(\de-rd\ka)}\ska{\nu,w}_{\R^2}\right)i\nu+o(1)\\
    &= \left(\frac{r\ka\de^2(1-2d\ka)+o(r)}{(\de-rd\ka)(rd+o(r))}+\ka^2\de\right)\frac{\ska{i\nu,w}_{\R^2}}{1-d\ka}\nu\\
    &\hspace{40pt}+\frac{\de\ka'}{1-d\ka}i\nu\ska{i\nu,w}_{\R^2}-\frac{\de}{d^2}i\nu\ska{\nu,w}_{\R^2}+o(1)\\
    &=\frac{\ka\de}{d}\nu\ska{i\nu,w}_{\R^2}+\frac{\de\ka'}{1-d\ka}i\nu\ska{i\nu,w}_{\R^2}
         -\frac{\de}{ d^2}i\nu\ska{\nu,w}_{\R^2}+o(1).
\end{aligned}
$$
This holds again uniformly for $u\in B_{\de_1}(\cM)$. Similar computations show that
$$
  \pa_rF_1(r,u) = \frac12\ka'd\nu-\frac12(1-d\ka)\ka i\nu+o(1)
$$
uniformly. Therefore $F$ is $\cC^1$ on all of $[0,r_0)\times B_{\de_1}(\cM)$ with $D_uF(0,u)[w]=\dot{w}+D_uF_1(0,u)[w]$ and $\pa_rF(0,u)=\pa_rF_1(0,u)$.
\end{proof}

\section{Proof of Theorem~\ref{thm:main}}\label{sec:solving}

Every element $v\in\cM=S^1*u^0$ is a solution of \eqref{eq:HS0} and hence of
$$
  F(0,v) = \dot{v}-\big(1-d(v)\ka(p(v))\big)\frac{\de}{d(v)} i\nu(p(v)) = 0.
$$
We shall prove that $\cM$ is a nondegenerate orbit of solutions of $F(0,\cdot)=0$. Using this nondegeneracy together with the underlying variational structure of the Hamiltonian system we will solve $F(r,u)=0$ for $r>0$ small and $u$ close to $\cM$.

Let $T_v\cM,N_v\cM\subset H^1$ denote the tangent and normal space of $\cM$ at $v\in\cM$ respectively. We also write $w^\perp\subset L^2$ for the orthogonal complement of $w\in L^2$, as usual.

\begin{Lem}\label{lem:M-nondeg}
For any $v\in\cM$ there holds
$$
  \Ker D_u F(0,v) =T_v\cM=\R\dot{v},\quad \Range D_u F(0,v)=\nu(p(v))^\perp.
$$
\end{Lem}

\begin{proof}
Clearly $T_v\cM=\R\dot{v}$ for $v\in\cM$. The $S^1$-equivariance of $F(0,\cdot)$ implies
$$
  D_uF(0,v)[\dot{v}] = \frac{d}{d\theta}_{|\theta=0}F(0,\theta*v)=0.
$$
It remains to prove $\Ker D_u F(0,v)\subset T_v\cM$. Suppose $w\in H^1$ satisfies $D_uF(0,v)[w]=0$. Since
$\dot{v}(t)\in \R i\nu(p(v(t)))$ for every $t\in\R$, we can decompose $w$ as
$$ w(t) = s(t)\dot{v}(t)+q(t)\nu(p(v(t))) $$
with $L$-periodic functions $s,q\in H^1_L(\R,\R)$. Applying Lemma~\ref{lem:FC1} and using $F(0,v)=0$, $d(v)\equiv \de$ we have
\begin{equation}\label{eq:kernel_of_DuF}
\begin{aligned}
  D_u F(0,v)[w]&=\dot{s}\dot{v}+sD_uF(0,v)[\dot{v}]+\dot{q}\nu(p(v))+qD_uF(0,v)[\nu(p(v))]\\
    &=\dot{s}\dot{v}+\dot{q}\nu(p(v))+q\frac{\ka(p(v))}{1-\de\ka(p(v))}\dot{v}-\frac{q}{\de}i\nu(p(v))\\
    &=\frac{\dot{v}}{\de}(\de\dot{s}-q)+\dot{q}\nu(p(v)).
\end{aligned}
\end{equation}
Thus $D_uF(0,v)[w] = 0$ implies $q(t)\equiv q\in\R$ and then $\dot{s}=\frac{q}{\de}\in\R$. Together with the periodicity of $s$, we see that $s$ must be a constant and thus $q=0$. Hence, $w\in \Ker D_uF(0,v)$ has the form $w=s\dot{v}$, which means $w\in T_v\cM$.

It remains to prove $\Range D_uF(0,v)=\nu(p(v))^\perp$. To do this we look at the operator $iD_uF(0,v)$, which splits as
$$
  iD_uF(0,v) = S+A
$$
with a symmetric operator
$$
  Sw = i\dot{w}+\frac{1}{\de}\nu\ska{\nu,w}_{\R^2}+\ka i\nu\ska{i\nu,w}_{\R^2}
$$
and a non-symmetric operator
$$
  Aw = -\frac{\ka'\de}{1-\de\ka}\nu\ska{i\nu,w}_{\R^2}.
$$
Observe that $\ska{Aw,\tilde{w}}_{L^2} = \ska{w,iAi\tilde{w}}_{L^2}$ for all $w,\tilde{w}\in H^1$. Hence for the adjoint of $iD_uF(0,v)$ there holds
$$
  (iD_uF(0,v))^* = S+iAi = iD_uF(0,v)+iAi-A.
$$
Now \eqref{eq:kernel_of_DuF} shows that $w = s\dot{v}+q\nu \in\Ker(iD_uF(0,v))^*$ if and only if
$$
\begin{aligned}
  0 = i\left(\frac{\dot{v}}{\de}(\de\dot{s}-q)+\dot{q}\nu\right)-i\frac{\ka'\de}{1-\de\ka}(q\nu+s\dot{v}),
\end{aligned}
$$
which is equivalent to the system
\begin{equation}\label{eq:qs_system}
  \begin{cases}
    \dot{q}=a(t)q,\\
    \dot{s}=a(t)s+\frac{q}{\de},
  \end{cases}
\end{equation}
where
$$
  a(t) = \frac{\ka'\de}{1-\de\ka} = \de D(\ka\circ p)(v)[i\nu] = \frac{\de}{1-\de\ka}D(\ka\circ p)(v)[\dot{v}]
       = -\frac{d}{dt}\log\big(1-\de\ka(p(v))\big).
$$
So the general solution of \eqref{eq:qs_system} is given by
$$
  q = c_q\frac{1}{1-\de\ka(p(v))},\quad\quad s = c_s\frac{1}{1-\de\ka(p(v))}+c_q\frac{1}{\de(1-\de\ka(p(v)))}t
$$
with constants $c_q,c_s\in\R$. Imposing $L$-periodicity on $s$ yields $c_q=0$ and therefore $w\in\Ker(iD_uF(0,v))^*$ if and only if
$$
  w=c_s\frac{1}{1-\de\ka(p(v))}\dot{v}=c_s i\nu(p(v)).
$$
It follows that $\Ker(iD_uF(0,v))^* = \R i\nu(p(v))$, hence $\Range D_uF(0,v) \subset \nu(p(v))^\perp$.

To conclude equality note that $F(0,u)-\dot{u} \in H^1$ for all $u\in B_{\de_1}(\cM)$, due to the regularity of $\ka,p$ and $\nu$. Thus if we define $P_0:L^2\to L^2$ to be the orthogonal projection onto the space of constant functions and the isomorphism $I:H^1\to L^2$, $u\mapsto \dot{u}+P_0u$, then $I^{-1}D_uF(0,v):H^1\to H^1$ is a compact perturbation of identity and hence an index $0$ Fredholm operator. Since $I$ is an isomorphism, the same holds for $D_uF(0,v)$. So
$$
  \codim \Range D_uF(0,v) = \dim\Ker D_uF(0,v)=1
$$
and therefore $\Range D_uF(0,v)=\nu(p(v))^\perp$.
\end{proof}

Since we are working on the tubular neighborhood $B_{\de_1}(\cM)$, it is enough to fix $v\in\cM$, e.g. $v=u^0$, and to solve $F(r,v+\cdot)=0$ on $B_{\de_1}(0)\cap N_v\cM$. For the fixed $v$ let $P_{\nu(p(v))}:L^2\to \R\nu(p(v))$ be the $L^2$-orthogonal projection onto $\R\nu(p(v))$.

\begin{Prop}\label{prop:IFT}
  There exist positive constants $r_1<r_0$, $\de_2<\de_1$ and a continuously differentiable map $[0,r_1)\to B_{\de_2}(0)\cap N_v\cM$, $r\mapsto w^{(r)}$, such that on $[0,r_1)\times B_{\de_2}(0)\cap N_v\cM$ holds
  $$
    (\id-P_{\nu(p(v))})F(r,v+w)=0\quad\Longleftrightarrow\quad w=w^{(r)}.
  $$
  In particular, $w^{(r)}=o(1)$ as $r\to0$.
\end{Prop}

\begin{proof}
This follows directly from lemma~\ref{lem:M-nondeg} and the implicit function theorem.
\end{proof}

It remains to solve the equation $F(r,u)=0$ on a one-dimensional subspace that is transversal to $\nu(p(v))^\perp$. In order to find such a space we use that $F$ is almost the derivative of an $S^1$-invariant functional. Note that we lost the variational structure of the original equation \eqref{eq:HS} due to the scaling with $\chi_r$, but the next lemma shows that there is still a natural direction in which $F$ vanishes.

\begin{Lem}\label{lem:natural_direction}
  For $(r,u)\in[0,r_0)\times B_{\de_1}(\cM)$ define the scalar, $L$-periodic and continuous function
  $$
    \la(r,u)=\frac{\de-rd(u)\ka(p(u))}{\de^2(1-d(u)\ka(p(u)))}.
  $$
  Then
  $$
    \ska{F(r,u),\la(r,u)i\dot{u}}_{L^2}=0.
  $$
\end{Lem}

\begin{proof}
With $H_r(u_1,\ldots,u_N):=\ham(\chi_r(u_1),\ldots,\chi_r(u_N))$ and $U:=(u,\theta_2*u,\ldots,\theta_N*u)$ as in Section~\ref{sec:scal} one can write $F$ as
$$
\begin{aligned}
  F(r,u) &= \dot{u}+rD\chi_r(u)^{-1}i\left(D\chi_r(u)^T\right)^{-1}\nabla_1 H_r(U)\\
         &= \dot{u}+r\det D\chi_r(u)^{-1}i\nabla_1 H_r(U)\\
         &= \dot{u}+\frac{1}{\la(r,u)}i\nabla_1 H_r(U).
\end{aligned}
$$
By \eqref{eq:symmetry_of_nabla_H} one has
$$
\begin{aligned}
  \ska{\nabla_kH_r(U),\theta_k*\dot{u}}_{L^2}
    &=\ska{\nabla_1H_r(\si^{-(k-1)}*U),\theta_k*\dot{u}}_{L^2} = \ska{\nabla_1H_r(\theta_k*U),\theta_k*\dot{u}}_{L^2}\\
    &=\ska{\nabla_1H_r(U),\dot{u}}_{L^2}
\end{aligned}
$$
Combining this yields
$$
  \ska{F(r,u),\la(r,u)i\dot{u}}_{L^2} = \ska{\nabla_1H_r(U),\dot{u}}_{L^2} = \frac{1}{N}\int_0^T DH_r(U)[\dot{U}]\:dt.
$$
In order to see that the last term vanishes look at
$$
  \phi_u:S^1\to\R,\quad\phi_u(\theta) = \int_0^TH_r(\theta*U)\:dt,
$$
which is in fact constant. If $u$ is $C^1$ one concludes
$$
  0=\phi_u'(0)=\int_0^TDH_r(U)[\dot{U}]\:dt,
$$
and by density the equation remains valid for every $u\in B_{\de_1}(\cM)$.
\end{proof}

Let $v\in\cM$ still be fixed as in Proposition \ref{prop:IFT}.

\begin{Lem}\label{lem:projections}
  There exists $r_2\in(0,r_0)$ and $\de_3\in(0,\de_1)$ such that
  $$
    L^2=\la(r,u)i\dot{u}\oplus\nu(p(v))^\perp
  $$
  for any $(r,u)\in[0,r_2)\times B_{\de_3}(v)$.
\end{Lem}

\begin{proof}
The statement is true as long as the map $T:[0,r_0)\times B_{\de_1}(\cM)\to\R$ defined by
$$
  T(r,u) = \ska{\la(r,u)i\dot{u},\nu(p(v))}_{L^2}
$$
is different from $0$. This is the case for $(r,u)$ close to $(0,v)$, because $T$ is continuous and
$$
  T(0,v) = \ska{\frac{1}{\de(1-\de\ka(p(v)))}i\dot{v},\nu(p(v))}_{L^2}
         = -\frac{1}{\de}\norm{\nu(p(v))}_{L^2}^2=-\frac{L}{\de}<0.
$$
\end{proof}

Now we define $\bar{r} := \min\{r_1,r_2\}$. Combining \ref{prop:IFT}, \ref{lem:natural_direction}, \ref{lem:projections} and taking $u^{(r)} = v+w^{(r)}$ we find for every $r\in[0,\bar{r})$ exactly one $S^1$-orbit $S_r = S^1*u^{(r)}$ contained in $B_{\min\set{\de_2,\de_3}}(\cM)$ of solutions of $F(r,\cdot)=0$. Moreover the continuum of solutions is parametrized over $[0,\bar{r})$ in a $\cC^1$ way, i.e.\ $r\mapsto u^{(r)}$ is $\cC^1$. Additionally one has $S_0 = \cM = S^1*u^0$. Finally, scaling back we obtain for $r\in(0,\bar{r})$ a $2\pi rL$-periodic solution of the original system \eqref{eq:HS} by setting
\[
  z_1^{(r)}(t) = \chi_r\big(u^{(r)}(t/(2\pi r))\big)\quad \textrm{and}\quad
  z_k^{(r)}(t) = z_1^{(r)}\left(t+\frac{2\pi rL(k-1)}{N}\right),\quad k=2,\cdots,N.
\]
Then all the properties in Theorem~\ref{thm:main} follow by our construction and Lemma~\ref{lem:r_derivative_of_u} below, where the parametrization of $\Ga$ by arc-length in \ref{thm:main}(2),(3) is given by $\ga = p\circ u^{(0)}$ and the rescaled function by $v^{(r)} = z_1(2\pi r\cdot) = \chi_r\big(u^{(r)}\big)$. Observe that $v^{(r)}$ satisfies
$$
  \dot{v}^{(r)} = -2\pi ri\nabla_{z_1}\ham\big(v^{(r)},\ldots,\theta_N*v^{(r)}\big),
$$
such that $\dot{v}^{(r)}=(1-r\ka)i\nu+o(r)$ not only holds in $L^2$, but also uniformly in $t\in\R$.

\begin{Lem}\label{lem:r_derivative_of_u}
  The $r$-derivative of $u^{(r)}$ at $r=0$ is given by $\displaystyle \pa_ru^{(0)} = -\frac{\de}{2}\ka\nu$.
\end{Lem}

\begin{proof}
Let $v:=u^{(0)}$. As in equation \eqref{eq:kernel_of_DuF} we write $\pa_rv = s\dot{v}+q\nu$ with periodic functions $s,q\in H^1_L(\R,\R)$. Differentiating of $F(r,u^{(r)})=0$ and using Lemma~\ref{lem:FC1} and \eqref{eq:kernel_of_DuF} then gives
$$
  \dot{q}\nu+\frac{\dot{v}}{\de}(\de\dot{s}-q) = D_uF(0,v)\pa_ru^{(0)}
    = -\pa_rF(r,v)=-\frac{\de}{2}\ka'\nu+\frac{1}{2}\ka \dot{v}.
$$
So with $q(0)=q_0\in\R$ we have $q=q_0-\frac{\de}{2}\ka$ and $\dot{s}=\frac{1}{2}\ka+\frac{q}{\de}=\frac{q_0}{\de}$. The periodicity of $s$ implies $q_0=0$ and $s\equiv s_0\in\R$. Finally $s_0=0$ is a consequence of $\pa_ru^{(0)}\in N_v\cM$. Indeed we have:
$$
\begin{aligned}
  0 = \ska{s_0\dot{v}-\frac{\de}{2}\ka\nu,\dot{v}}
    = s_0\norm{\dot{v}}^2-\frac{\de}{2}\int_0^L\ska{\frac{d}{dt}(\ka\nu),\ddot{v}}_{\R^2}\:dt
    = s_0\norm{\dot{v}}^2+\frac{\de}{2}\int_0^L\ka\dot{\ka}\:dt = s_0\norm{\dot{v}}^2.
\end{aligned}
$$
\end{proof}

\section{Proof of Proposition~\ref{prop:simply_connected}}\label{sec:simply_connected_domains}

Consider a bounded and simply connected domain $\Om\subset\C$ with $\cC^{3,\alpha}$ boundary. Let
$$
  G(x,y) = -\frac{1}{2\pi}\log\abs{x-y}-g(x,y)
$$
be the Green function for the Dirichlet Laplace operator, and $\rho:\Om\to\R$ the corresponding conformal radius, which is  defined by the relation
\[
  -\frac{1}{2\pi}\log(\rho(z)) = g(z,z).
\]
We proof that the pair $\rho,G$ satisfies Assumption \ref{ass}.

\begin{proof}[Proof of Proposition \ref{prop:simply_connected}] The part concerning $\rho$ has been proved by Bandle and Flucher, \cite{flucher:1999,bandle-flucher:1996}. They showed that the conformal radius in terms of a Riemann map $f:\Om\to B_1(0)$ is given by
\begin{equation}\label{eq:conformal_radius_riemann_map}
  \rho = \frac{1-\abs{f}^2}{\abs{f'}}.
\end{equation}
Since $\pa\Om\in\cC^{3,\alpha}$, $f$ can be extended to $f\in\cC^{3,\alpha}(\overline{\Om})$ (cf. proof of theorem 3.1 in \cite{bell:1990}). In order to see  $\rho\in\cC^3(\overline{\Om})$, derive \eqref{eq:conformal_radius_riemann_map} three times in the interior of $\Om$ and observe that the fourth derivative of $f$ only appears in one term containing the product $\rho f^{(4)}$. But for $z\to\pa\Om$ there holds
$$
\begin{aligned}
  \abs{\rho(z)f^{(4)}(z)} = \frac{\rho(z)}{d(z)}\cdot\abs{\frac{d(z)}{2\pi i}\int_{\abs{w-z}
     = \frac{d(z)}{2}}\frac{f^{(3)}(w)-f^{(3)}(z)}{(w-z)^2}\:dw} = O(d(z)^\alpha),
\end{aligned}
$$
because $\rho=2d+o(d)$ (see \cite{flucher:1999,bandle-flucher:1996}). Therefore $\rho\in\cC^3(\overline{\Om})$. Moreover Bandle and Flucher have shown that
\begin{equation*}\label{eq:flucherexpansion}
  \rho(p-d\nu(p)) = 2d-\ka(p)d^2+o(d^2)
\end{equation*}
for every boundary point $p\in\pa \Om$. Thus one obtains on the boundary $\pa\Om$:  $\rho=0$, $\ska{\nabla \rho,\nu}=-2$ and $\ska{\rho''\nu,\nu}=-2\ka$. Clearly $\ska{\nabla\rho,i\nu}=0$, hence $\rho''i\nu = D(-2\nu)[i\nu] = -2\ka i\nu$ shows $\rho'' = -2\ka\cdot\id$ on $\pa\Om$.

We first check the properties of $G$ when $\Om$ is the unit disc. In that case one has
\begin{equation}\label{eq:Greens_function_ball}
  G_{B_1(0)}(x,y) = -\frac{1}{2\pi}\left(\log\abs{x-y}-\log\abs{x-R(y)}-\log\abs{y}\right),
\end{equation}
where $R(y):=\frac{y}{\abs{y}^2}$. For every $\eps>0$ there exist constants $c,\tilde{\eps}>0$, such that
$$
  \abs{\nabla_1G_{B_1(0)}(x,y)}
    = \abs{-\frac{1}{2\pi}\left(\frac{x-y}{\abs{x-y}^2}-\frac{x-R(y)}{\abs{x-R(y)}^2}\right)}
    \leq cd(y)
$$
holds for $(x,y)\in B_1(0)^2$ with $\abs{x-y}\geq \eps$ and $d(y)=1-\abs{y}<\tilde{\eps}$. The same is valid for
$$
\begin{aligned}
  \nabla_1^2G_{B_1(0)}(x,y)
    &= -\frac{1}{2\pi}\left(\frac{1}{\abs{x-y}^2}-\frac{1}{\abs{x-R(y)}^2}\right)\cdot\id\\
    &\hspace{1cm}
      +\frac{1}{2\pi}\left(2\frac{(x-y)(x-y)^T}{\abs{x-y}^4}-2\frac{(x-R(y))(x-R(y))^T}{\abs{x-R(y)}^4}\right).
\end{aligned}
$$
For $\nabla_2\nabla_1G_{B_1(0)}$ one observes for $(x,y)$ as above that
$$\begin{aligned}
  \nabla_2\nabla_1G_{B_1(0)}(x,y)
    &= -\nabla_1^2G_{B_1(0)}(x,y)+\nabla_2\nabla_1G_{\R^2}(x,R(y))\cdot(\id-DR(y))\\
    &= O(d(y))+O(1)\cdot(\id-DR(y)),
\end{aligned}
$$
where $G_{\R^2}(x,y)=-\frac{1}{2\pi}\log\abs{x-y}$. Now
$$
  \id-DR(y) = \left(1-\frac{1}{\abs{y}^2}\right)\cdot\id-2\frac{yy^T}{\abs{y}^4} = O(d(y))+O(1)\nu(p(y))^T
$$
implies the asymptotic behavior for $\nabla_2\nabla_1 G_{B_1(0)}$.

For the general case we use a Riemann map $f:\Om\to B_1(0)$, which again extends to the boundary $\cC^{4,\alpha}$ smooth, so that
$$
  G_\Om(x,y)=G_{B_1(0)}(f(x),f(y)).
$$
The result for $\nabla_1G_\Om$, $\nabla_1^2 G_\Om$ follows then from the properties of $G_{B_1(0)}$ and the fact that there exists a constant $c>0$ with $d(f(y))\leq cd(y)$. In the same manner one has
$$
\begin{aligned}
  \nabla_2\nabla_1G_\Om(x,y)
    &= O(d(y))+O(1)\nu(p(f(y)))^TDf(y)\\
    &= O(d(y))+O(1)\nu(f(p(y)))^TDf(p(y))),
\end{aligned}
$$
from which we can conclude the proposition because $f$ is a biholomorphic map.
\end{proof}


{\sc Thomas Bartsch\\
 Mathematisches Institut, Universit\"at Giessen\\
 35392 Giessen, Germany}\\
 Thomas.Bartsch@math.uni-giessen.de

\vspace{2mm}
{\sc Qianhui Dai\\
 College of Science, China University of Petroleum-Beijing\\
 102249 Beijing, China}\\
 qhdai@cup.edu.cn

\vspace{2mm}
{\sc Bj\"orn Gebhard\\
 Mathematisches Institut, Universit\"at Giessen\\
 35392 Giessen, Germany}\\
 Bjoern.Gebhard@math.uni-giessen.de\\[1em]

\end{document}